\documentclass[12pt]{amsart}

\usepackage{amsmath,amssymb}
\usepackage{amsfonts}
\usepackage{amsthm}
\usepackage{bm}

\newtheorem{theorem}{Theorem}[section]
\newtheorem{lemma}[theorem]{Lemma}

\newtheorem{example}[theorem]{Example}

\newcommand{\bb}[1]{\mathbf{#1}}

\newcommand{\tr}{\mathrm{tr}\,}
\newcommand{\nn}[1]{\left|#1\right|}
\newcommand{\Jac}{\mathrm{Jac}\,}
\newcommand{\pphi}{\bm{\phi}}
\newcommand{\ppsi}{\bm{\psi}}

\begin{document}
\raggedbottom

\title[Uniqueness of Planar Immersions]{Planar Immersions with Prescribed Curl and Jacobian Determinant are Unique}

\author{Anthony Gruber}
\address{Department of Scientific Computing, Florida State University, 400 Dirac Science Library, Tallahassee, FL 32306-4120.}
\email{agruber@fsu.edu}

\begin{abstract} 
We prove that immersions of planar domains are uniquely specified by their Jacobian determinant, curl function, and boundary values.  This settles the two-dimensional version of an outstanding conjecture related to a particular grid generation method in computer graphics.

\vspace{1pc}

\end{abstract}

\maketitle
% \tableofcontents

\section{Introduction}%%%%
The problem of generating two and three-dimensional meshes is of critical importance in scientific computing, where the performance of simulations relies heavily on well discretized domains.  A  noteworthy subset of techniques for this purpose are known as the ``variational methods'' (see e.g. \cite{steinberg1986,castillo1991,liao2009, garanzha2014, tingfan2021} and references therein), which aim to control characteristic quantities of grid transformations such as their divergence, curl, and Jacobian determinant.  In essence, such variational methods are numerical techniques for minimizing nonlinear and often nonconvex functionals whose critical points represent mappings with desirable properties.  Since the minimizers of such objects often have prescribed geometric data as a consequence, it is of both theoretical and practical interest to determine when a reasonable set of prescribed conditions is sufficient to determine a mapping which is unique.

To this end, research activity into grid generation methods for computing applications has led to a number of conjectures regarding the uniqueness of solutions to various boundary-value problems, which arise e.g. as the Euler-Lagrange equations of practically relevant functionals.  In \cite{chen2016,zhou2017}, the authors conjecture that prescribing the Jacobian determinant, curl function, and boundary values of a smooth mapping $\pphi: M^n \subset \mathbb{R}^n \to \mathbb{R}^n$ is sufficient for a unique solution when $n$ is two or three.  They present numerical evidence for this statement when $n=2$, as well as an argument showing that it holds in the special case when one mapping is identity and the other is sufficiently close to identity in $H^2_0$-norm.  Since that work, more numerical evidence has been presented in support of this conjecture when $n=2$, and also when $n=3$ and the mapping is identity on the boundary (see e.g. \cite{zhu2019,zhou2021}).  The purpose of this note is to confirm the uniqueness conjecture when $n=2$ and $\pphi:M\to\mathbb{R}^2$ is an immersion, i.e. when the differential $d\pphi:TM \to T\mathbb{R}^2|_{\pphi(M)}$ has full rank everywhere on $M$.  This claim will follow from Theorem~\ref{thm:main}, which will be proved following some preliminary discussion.  After the proof, brief remarks are offered on the result as well as the open case of $n=3$.

\section{Preliminaries}
Let $\bb{v}\cdot\bb{v} = |\bb{v}|^2$ denote the standard (Euclidean) inner product on vectors $\bb{v}\in\mathbb{R}^2$.  Consider the Clifford algebra $Cl(\mathbb{R}^2,|\cdot|^2) = \mathcal{G}_2$ (sometimes called the geometric algebra of Euclidean 2-space), which is the real $2^2$-dimensional quotient of the free tensor algebra on $\mathbb{R}^2$ by the two-sided ideal $\bb{v}\otimes \bb{v} - \nn{\bb{v}}^21$.  We briefly discuss some properties of this object which will be useful for the computations in the body.  Further discussion can be found e.g. in \cite{hestenes1984, lounesto2001}.

The algebra $\mathcal{G}_2$ is generated by the symbols $\{1, \bb{e}_1, \bb{e}_2\}$ where the $\bb{e}_i$ diagonalize the inner product and satisfy the relations
\[\bb{e}_1^2 = \bb{e}_2^2 = 1,\qquad  \bb{e}_1 \cdot\bb{e}_2 = 0, \qquad \bb{e}_1\bb{e}_2 = -\bb{e}_2\bb{e}_1. \]
In particular, the product operation in $\mathcal{G}_2$ (called the geometric product in \cite{hestenes1984}) is not commutative, and relates the inner product $\bb{v}\cdot\bb{w}$ with an outer product $\bb{v}\wedge\bb{w}$ on vectors $\bb{v},\bb{w} \in \mathrm{Span}\{\bb{e}_i\} = \mathbb{R}^2$ through its symmetric and antisymmetric parts,
\begin{align*}
    \bb{v}\bb{w} + \bb{w}\bb{v} &= 2\bb{v}\cdot\bb{w}, \\
    \bb{v}\bb{w} - \bb{w}\bb{v} &= 2\bb{v}\wedge\bb{w}.
\end{align*}
Adding these relations yields the useful formula
\begin{equation}\label{eq:product}
    \bb{v}\bb{w} = \bb{v}\cdot \bb{w} + \bb{v}\wedge \bb{w},
\end{equation}
which relates the inner, outer, and geometric products.  Note that the ternary expression $(\bb{u}\bb{v})\bb{w} = \bb{u}(\bb{v}\bb{w}) = \bb{u}\bb{v}\bb{w}$ is associative, but the same is not true of expressions like  $(\bb{u}\cdot\bb{v})\bb{w}$ and $(\bb{u}\wedge\bb{v})\bb{w}$.

Elements of $\mathcal{G}_2$ are $\mathbb{R}$-linear combinations of $k$-vectors where $k=0,1,2$ is known as the grade.  As such, each element $A \in \mathcal{G}^2$ has the unique decomposition as a sum of scalar, (0-vector), vector (1-vector), and bivector (2-vector) parts,
\[ A = \alpha + \bb{a} + J\beta = \sum_{k=0}^2 \langle A\rangle_k, \]
where $\alpha,\beta \in \mathbb{R}$, $\bb{a}\in\mathbb{R}^2$, $\langle A\rangle_k$ is the grade operator which selects the $k$-vector part of $A$, and $J = \bb{e}_1\bb{e}_2 = \bb{e}_1\wedge\bb{e}_2$ is the unit pseudoscalar of the algebra $\mathcal{G}_2$.  Notice that $J^2 = -1$ and $\bb{v}J = -J\bb{v}$ for all $v\in\mathbb{R}^2$.  Combined with \eqref{eq:product}, this yields the essential duality relations
\begin{equation}\label{eq:duality}
    \begin{split}
        J(\bb{v} \cdot \bb{w}) &= \bb{v} \wedge (J\bb{w}), \\
        J(\bb{v}\wedge\bb{w}) &= \bb{v} \cdot (J\bb{w}).
    \end{split}
\end{equation}
In particular, observe that bivectors are dual to scalars while vectors are self-dual.  Since $J\bb{e}_1 = -\bb{e}_2$ and $J\bb{e}_2 = \bb{e}_1$, it is conceptually useful to view the pseudoscalar $J$ as a rotation clockwise by $\pi/2$ radians when left-acting on vectors.

% Note that in the remainder of this note, the adjective ``smooth'' is used to denote as much regularity as necessary for the arguments to hold (usually $C^1$).

% This innocuous but powerful fact will enable the arguments of the next Section.
% In practice, it is customary to associate $\{\bb{e}_1,\bb{e}_2\}$ with the standard vector basis for $\mathbb{R}^2$, and. 
% There is a unit pseudoscalar $J = \bb{e}_1 \wedge \bb{e}_2$ satisfying $\bb{v}J = -J\bb{v}$ for all $v\in\mathcal{G}^2$.  

\section{Proof of the Conjecture}
The main result relies on two basic facts, the first of which is that the induced Riemannian metric is computable from the Jacobian determinant and curl function.  We remark that for the remainder of this note, the adjective ``smooth'' is used to denote as much regularity as necessary for the arguments to hold ($C^1$ is likely sufficient).

\begin{lemma}\label{lem:samemetric}
Suppose $\pphi: M \to \mathbb{R}^2$ is a smooth immersion of the two-dimensional manifold $M$.  Then, the Riemannian metric $d\pphi\cdot d\pphi$ induced on $M$ by $\pphi$ is computable from the Jacobian determinant $\Jac \pphi$ and the curl function $\nabla \times \pphi$. 
\end{lemma}

\begin{proof}
Consider the dual mapping $J\pphi: M \to \mathbb{R}^2$, and notice that its differential $d(J\pphi) = J\,d\pphi$ has full rank since $\pphi$ is an immersion and $J$ has trivial kernel.  Then, using the duality relations \eqref{eq:duality} and the fact that $\pphi_1\wedge\pphi_2 = J\,\det(d\pphi)$, we compute 
\begin{align*}
    \Jac J\pphi &= -J\left((J\pphi)_1 \wedge (J\pphi)_2\right) = (J\pphi_1) \cdot \pphi_2 = J\left(\pphi_2 \wedge \pphi_1\right) = \Jac \pphi.
\end{align*}
Moreover, knowledge of the curl is equivalent to
\[ \nabla\times\pphi = J\left(\nabla\wedge\pphi\right) = \nabla \cdot (J\pphi) = \tr d(J\pphi). \]
Therefore, the prescribed Jacobian determinant condition carries over to the dual immersion $J\pphi$, and the prescribed curl condition translates to a prescribed trace condition on the differential $d(J\pphi)$.  Since $d(J\pphi)$ has rank two everywhere, this information determines its characteristic polynomial at each point of $M$, hence also its eigenvalues (where we consider $T\mathbb{R}^2|_{J\pphi(M)} \cong TM$ since $J\pphi$ is an isometry for the induced metric on $M$).

% \[ \IP{}{J\bb{v}}{J\bb{w}} = \left\langle J\bb{v}J\bb{w}\right\rangle_0 = \left\langle \bb{v}\bb{w}\right\rangle_0 = \IP{}{\bb{v}}{\bb{w}},\]
% \[ \IP{}{d(J\pphi)}{d(J\pphi)} = \IP{}{J\,d\pphi}{J\,d\pphi} = \IP{}{d\pphi}{d\pphi}. \]

Now, recall that $J\bb{v} = -\bb{v}J$ for all $\bb{v}\in \mathbb{R}^2$, so that
\[ (J\bb{v})\cdot(J\bb{w}) = \left\langle J\bb{v}J\bb{w}\right\rangle_0 = \left\langle \bb{v}\bb{w}\right\rangle_0 = \bb{v}\cdot\bb{w},\]
for any $\bb{v},\bb{w} \in \mathbb{R}^2$.  It follows that $-J$ is the adjoint of $J$ with respect to the Euclidean inner product, and hence
\[ d(J\pphi)\cdot d(J\pphi) = (J\,d\pphi)\cdot (J\,d\pphi) = d\pphi\cdot d\pphi. \]
It remains to show that at least one of the above expressions is computable from the eigenvalues of $d(J\pphi)$.  To that end, consider the eigenvalue problem $d(J\pphi)(\bb{v}) = \lambda \bb{v}$.  Knowledge of the characteristic polynomial implies this can be solved around any point $x\in M$ for relevant pairs $(\lambda_i, \bb{f}_i)$ which may be real or complex conjugates.  Following this, it is straightforward to compute the one-form basis $\{\omega^1, \omega^2\}$ for the cotangent space $T^*M$ satisfying $\omega^i(\bb{f}_j) = \delta^i_j$, so that the differential of $J\pphi$  diagonalizes as $d(J\pphi) = \sum \lambda_i\,\omega^i\otimes \bb{f}_i$.  Moreover, we have that $d(J\pphi) = J\,d\pphi$, so if $\lambda_i$ is an eigenvalue of $d(J\pphi)$ with eigenvector $\bb{f}_i$ then the equation $d\pphi(\bb{f}_i) = \lambda_i(-J\bb{f}_i)$ is also satisfied.  Therefore, expanding arbitrary vectors $\bb{v},\bb{w} \in TM$ in the eigenvector basis as $\bb{v} = \sum v^i\,\bb{f}_i$ and $\bb{w} = \sum w^i\,\bb{f}_i$, it follows that 
\[ d\pphi(\bb{v})\cdot d\pphi(\bb{w}) = \sum_{i,j} \lambda_i\,v^i(-J\bb{f}_i)\cdot\lambda_j\,w^j(-J\bb{f}_j) = \sum_{i,j} \lambda_i\lambda_j\,v^iw^j\,\bb{f}_i\cdot\bb{f}_j, \]
which is computable.  Hence, we conclude that the hypothesized information is sufficient for determining the metric on $M$ induced by $\pphi$, completing the argument.
\end{proof}

% On the other hand, $d(J\pphi) = J\,d\pphi$, so that if $\lambda$ is an eigenvalue of $d(J\pphi)$ with eigenvector $\bb{v}$ then $\lambda$.....  Therefore, we can directly solve $J\,d\pphi(\bb{v}) = \lambda(-J\bb{v})$ by exploiting the relation $\det(\lambda J + d\pphi) = \det(J)\det(\lambda I - J\,d\pphi) = \det(\lambda I - J\,d\pphi)$. Doing this gives access to the eigenvalues $\lambda_1, \lambda_2$, which can be used to form the eigenvector basis $\{\bb{f}_1, \bb{f}_2\}$ for $d\pphi$ as well as the one-form basis $\{\omega^1, \omega^2\}$ satisfying $\omega^i(\bb{f}_j) = \delta^i_j$.  Collecting this information, the differential diagonalizes as  $d\pphi = \sum \lambda_i\,\omega^i\otimes \bb{f}_i$, so that if $\bb{v} = \sum v^i\,\bb{f}_i$ we have
% \[ d\pphi(\bb{v})\cdot d\pphi(\bb{w}) = \sum_{i,j} \lambda_i\lambda_j\,v^iw^j\,\bb{f}_i\cdot\bb{f}_j, \]
% for any $\bb{v},\bb{w}\in TM$.  Hence, we conclude that the hypothesized information is sufficient for computing the metric on $M$ induced by $\pphi$, completing the argument.

The above result is perhaps best illustrated with a simple example.

\begin{example}
Consider $\pphi(x,y) = \begin{pmatrix}x & y\end{pmatrix}^\top$ which is an immersion of $\mathbb{R}^2$.  Then, the dual immersion is $J\pphi(x,y) = \begin{pmatrix}y & -x\end{pmatrix}^\top$, whose differential is
\[ d(J\pphi) = -dx\otimes\bb{e}_2 + dy\otimes\bb{e}_1 = dx\otimes J\bb{e}_1 + dy \otimes J\bb{e}_2 = J\,d\pphi. \]
It is straightforward to check that $\tr d(J\pphi) = 0$ and $\det d(J\pphi) = 1$, so that $d(J\pphi)$ has eigenvalues $\lambda_k = \pm i \in \mathbb{C}$, with corresponding eigenvectors $\bb{f}_k = \begin{pmatrix}\pm i & 1\end{pmatrix}^\top$ and dual one-forms $\omega^k = (1/2)\left( \pm idx + dy \right)$.  Suppose we are given this information without any prior knowledge of the immersions $\pphi, J\pphi$.  As in the proof of Lemma~\ref{lem:samemetric}, for any $\bb{v} = \sum v^k\,\bb{e}_k = \sum \tilde{v}^k\,\bb{f}_k \in \mathbb{R}^2$, we verify
\[d\pphi(\bb{v}) = \bb{v} = i\,\frac{iv^1 + v^2}{2} \begin{pmatrix}-1 \\ -i\end{pmatrix} - i\,\frac{-iv^1 + v^2}{2} \begin{pmatrix}-1 \\ i\end{pmatrix} = i\tilde{v}^1(-J\bb{f}_1) - i\tilde{v}^2(-J\bb{f}_2). \]
Using this, it is straightforward to calculate
\begin{align*}
    d\pphi(\bb{v})\cdot d\pphi(\bb{w}) &= \sum_{i,j}\lambda_i\lambda_j\,\tilde{v}^i\tilde{w}^j\,\bb{f}_i \cdot \bb{f}_j = -2i^2(\tilde{v}^1\tilde{w}^2 + \tilde{v}^2\tilde{w}^1) \\
    &= \frac{(iv^1 + v^2)(-iw^1 + w^2)}{2} + \frac{(-iv^1 + v^2)(iw^1 + w^2)}{2} = v^1w^1 + v^2w^2,
\end{align*}
which is the Euclidean metric on the domain.  It is easy to see that this agrees with the expression generated by direct calculation from the definition of $\pphi$.
\end{example}

% (resp. $\lambda_1, \bar{\lambda}_1 \in \mathbb{C}$)
% (resp. $\{\bb{f}_1, \bar\bb{f}_1\}$)

% hence also its (possibly complex) eigenvalues with respect to the (possibly complex) eigenvector basis $\{\bb{E}_1,\bb{E}_2\}$ for $TM \cong f^*(T\mathbb{R}^2)$.

% In fact, access to the eigenvalues of $d(J\pphi)$ provides a way to compute the metric on $M$ induced by $\pphi$.

% if the eigenvalues of $d(J\pphi)$ are known, we can solve the eigenvalue problem $d\pphi(\bb{v}) = \lambda \bb{v}$ for a basis $\{\bb{f}_1,\bb{f}_2\}$ corresponding to scalar values $\lambda_1,\lambda_2$. Note that $d(J\pphi) = J\,d\pphi$ always has real entries with respect to $\bb{e}_1, \bb{e}_2$, so that $\lambda_2, \bb{f}_2 = \bar{\lambda}_1, \bar{\bb{f}}_1$ if $\lambda_1\in \mathbb{C}$.  With this, it is straightforward to compute the one-form basis $\{\omega^1, \omega^2\}$ satisfying $\omega^i(\bb{f}_j) = \delta^i_j$, so that the differential is diagonalized as $d(J\pphi) = \sum \lambda_i\,\omega^i\otimes\bb{f}_i$.  Then, we further compute 
% \[ d(J\pphi) = J\,d\pphi = \sum \lambda_i\,\omega^i\otimes(-J^2)\bb{f}_i = \sum \lambda_i\,-J\omega^i\otimes -J\bb{f}_i \]

% hence also the metric induced by $\pphi$.

% since the metric on $M$ induced by $J\pphi$ is computable from the given information, so is the metric induced by $\pphi$.

The next fact establishes equivalence of mappings in the case of boundary agreement and equivalence of induced metrics.  The proof is inspired by  \cite{kohanSE}.

\begin{lemma}\label{lem:samemapping}
Suppose $\pphi,\ppsi:M \to \mathbb{R}^n$ are smooth immersions of the n-dimensional manifold $M$ with boundary $\partial M$.  If $\pphi,\ppsi$ induce the same Riemannian metric $g = d\pphi\cdot d\pphi = d\ppsi\cdot d\ppsi$ on $M$ and $\pphi|_{\partial M} = \ppsi|_{\partial M}$, then $\pphi \equiv \ppsi$.
\end{lemma}

\begin{proof}
Consider an interior point $x\in M$ and a tangent vector $\bb{v}\in T_xM$.  Endow $M$ with the metric $g$ and choose the maximal unit-speed geodesic $\gamma: I \subset \mathbb{R} \to M$ such that $\gamma(0) = x$ and $\gamma'(0) = \bb{v}$.  We claim that $I = [a,b]$ is a closed interval for some $a,b\in\mathbb{R}$.  To see this, first note that $I$ cannot be unbounded.  Indeed, $\pphi$ is a Euclidean isometry, therefore $\pphi\circ\gamma$ is a constant-speed parameterized Euclidean line segment contained in the compact set $\pphi(M) \subset \mathbb{R}^n$.  It follows that the image of $\pphi\circ\gamma$ exits $\pphi(M)$ at two points $\pphi(p), \pphi(q)$ necessarily on $\pphi(\partial M)$, which correspond to parameter values $a,b\in\mathbb{R}$ such that $p=\gamma(a), q=\gamma(b) \in\partial M$.  Hence, the image of $\pphi\circ\gamma$ (resp. $\gamma$) is closed, implying that $I = [a,b]$ as claimed.  Returning to the main argument, the hypotheses now imply that $\pphi \circ \gamma$ and $\ppsi\circ\gamma$ are Euclidean line segments sharing endpoints.  Since a pair of points determines a unique Euclidean geodesic, we conclude that $\pphi = \ppsi$ along $\gamma$, and the result follows as $x \in M$ and $\bb{v}\in T_xM$ are arbitrary.
\end{proof}

The two-dimensional uniqueness conjecture is now stated and proved.

\begin{theorem}\label{thm:main}
Let $M \subset \mathbb{R}^2$ be a smooth two-dimensional manifold with boundary $\partial M$ and let $\pphi, \ppsi: M \to \mathbb{R}^2$ be smooth immersions satisfying
\begin{align*}
    \Jac\pphi &= \Jac\ppsi, \\
    \nabla\times\pphi &= \nabla\times\ppsi, \\
    \pphi|_{\partial M} &= \ppsi|_{\partial M}.
\end{align*}
Then, $\pphi \equiv \ppsi$ on $M$.
\end{theorem}

\begin{proof}
By Lemma~\ref{lem:samemetric}, the Jacobian determinant and curl hypotheses imply that $d\pphi \cdot d\pphi = d\ppsi \cdot d\ppsi$.  The boundary conditions and Lemma~\ref{lem:samemapping} then establish that $\pphi = \ppsi$ pointwise on $M$. 
\end{proof}

\section{Remarks on the Result}
Note that the conditions of prescribed Jacobian determinant $\Jac\pphi$, prescribed curl $\nabla\times\pphi$, and prescribed boundary values $\pphi|_{\partial M}$ already overdetermine the mapping $\pphi$ in the two-dimensional case.  In particular, a straightforward computation using the duality relations \eqref{eq:duality} and integration by-parts shows
\[ \int_M \nabla \times \pphi = \int_M \nabla\cdot (J\pphi) = \int_{\partial M} (J\pphi)\cdot \bb{n} = -\int_{\partial M} \pphi \cdot \bb{T}, \]
where $\{\bb{T},\bb{n}\}$ is a right-handed orthonormal frame for $\partial M$.  In particular, the curl and the boundary values must be compatible in order for a solution to exist at all.

Conversely, prescribing the Jacobian determinant and curl functions is highly unlikely to be sufficient for specifying unique immersions of three-manifolds with boundary immersed in $\mathbb{R}^3$.  Indeed, even if the arguments in Lemma~\ref{lem:samemetric} can be translated to this setting, these conditions are no longer sufficient to determine the induced metric on $M$.  For example, consider the simple case of $\pphi(x,y,z) = \begin{pmatrix}y & z & x\end{pmatrix}^\top$ and $\ppsi(x,y,z) = \begin{pmatrix}y-x & z & x\end{pmatrix}^\top$.  It is easily checked that $\Jac \pphi = \Jac \ppsi = 1$ and $\nabla \times \pphi = \nabla \times \ppsi = -\begin{pmatrix}1 & 1 & 1\end{pmatrix}^\top$, but 
\begin{align*}
    d\pphi \cdot d\pphi &= dx^2 + dy^2 + dz^2, \\
    d\ppsi \cdot d\ppsi &= 2dx^2 - 2dxdy + dy^2 + dz^2,
\end{align*} so that $\pphi$ induces the standard Euclidean metric while $\ppsi$ does not.  As the three-dimensional case is very relevant to computer graphics and volumetric mesh generation, it would be interesting to confirm or deny the uniqueness conjecture of \cite{zhou2017} in this case as well.  If it does not hold, is there a small set of conditions (aside from the classical pair of divergence and curl) which specify a unique mapping in this situation as well?  Finally, we note that the immersion condition on $\pphi$ used extensively in this note can potentially be weakened, although care will be required to handle the regions of $M$ where $d\pphi$ is rank-deficient.  On the other hand, many works (such as several mentioned in the Introduction) appear to be most interested in diffeomorphic solutions, so perhaps this degree of regularity is already satisfactory in practice.

% know a minimal set of conditions (aside from usual pair of divergence and curl) which uniquely determine a three-dimensional mapping.

% the metric on $M$ induced by $J\pphi$ is identical to the metric induced by $\phi$, which implies that the leftmost expression is real even when the eigenvalues of $d(J\pphi)$ are not.  

\bibliographystyle{abbrv}
\bibliography{biblio}

\begin{thebibliography}{10}

\bibitem{castillo1991}
J.~E. Castillo.
\newblock A discrete variational grid generation method.
\newblock {\em SIAM Journal on Scientific and Statistical Computing},
  12(2):454--468, 1991.

\bibitem{chen2016}
X.~Chen.
\newblock {\em Numerical Construction of Diffeomorphism and the Applications to
  Grid Generation and Image Registration}.
\newblock PhD thesis, 2016.

\bibitem{garanzha2014}
V.~A. Garanzha, L.~N. Kudryavtseva, and S.~V. Utyuzhnikov.
\newblock Variational method for untangling and optimization of spatial meshes.
\newblock {\em Journal of Computational and Applied Mathematics}, 269:24--41,
  2014.

\bibitem{hestenes1984}
D.~Hestenes and G.~Sobczyk.
\newblock {\em Clifford Algebra to Geometric Calculus: A Unified Language for
  Mathematics and Physics}.
\newblock Fundamental Theories of Physics. Springer Netherlands, 1984.

\bibitem{kohanSE}
M.~Kohan.
\newblock Pointwise equality of codimension-zero immersions.
\newblock Mathematics Stack Exchange.
\newblock URL:https://math.stackexchange.com/q/4205739 (version: 2021-07-23).

\bibitem{liao2009}
G.~Liao, X.~Cai, J.~Liu, X.~Luo, J.~Wang, and J.~Xue.
\newblock Construction of differentiable transformations.
\newblock {\em Applied Mathematics Letters}, 22(10):1543--1548, 2009.

\bibitem{lounesto2001}
P.~Lounesto.
\newblock {\em Clifford Algebras and Spinors}, volume 286.
\newblock Cambridge University Press, 2001.

\bibitem{steinberg1986}
S.~Steinberg and P.~J. Roache.
\newblock Variational grid generation.
\newblock {\em Numerical Methods for Partial Differential Equations},
  2(1):71--96, 1986.

\bibitem{tingfan2021}
T.~Wu, X.~Liu, W.~An, Z.~Huang, and H.~Lyu.
\newblock A mesh optimization method using machine learning technique and
  variational mesh adaptation.
\newblock {\em Chinese Journal of Aeronautics}, 2021.

\bibitem{zhou2017}
Z.~Zhou, X.~Chen, X.~X. Cai, and G.~Liao.
\newblock Uniqueness of transformation based on {J}acobian determinant and
  curl-vector.
\newblock {\em arXiv preprint arXiv:1712.03443}, 2017.

\bibitem{zhou2021}
Z.~Zhou and G.~Liao.
\newblock Construction of diffeomorphisms with prescribed {J}acobian
  determinant and curl.
\newblock {\em arXiv preprint arXiv:2105.09302}, 2021.

\bibitem{zhu2019}
Y.~Zhu, Z.~Zhou, G.~Liao, Q.~Yang, and K.~Yuan.
\newblock Effects of differential geometry parameters on grid generation and
  segmentation of {MRI} brain image.
\newblock {\em IEEE Access}, 7:68529--68539, 2019.

\end{thebibliography}

\end{document}